\documentclass[final,1p,times,authoryear]{elsarticle}

\usepackage{amsmath} 
\usepackage{bm}
\usepackage{amsthm}
\usepackage{hypernat}
\usepackage[hidelinks]{hyperref}
\usepackage[T1]{fontenc}
\usepackage{booktabs}

\usepackage{natbib}
\setcitestyle{numbers,square}

\everymath{\displaystyle}

\theoremstyle{definition}

\theoremstyle{plain}

\newtheorem{thm}{Theorem}[section]
 
 \newtheorem{lem}[thm]{Lemma}
 \newtheorem{prop}[thm]{Proposition}
 \theoremstyle{definition}
 \newtheorem{defn}[thm]{Definition}
 \theoremstyle{remark}

 \numberwithin{equation}{section}
\journal{Mathematical Methods in the Applied Sciences}

\begin{document}

\begin{frontmatter}



\title{The Extended Paley-Wiener Theorem over the Hardy-Sobolev Spaces}


\author[author1]{Detian Liu}
\author[author2]{Haichou Li}
\author[author3]{Kit Ian Kou}
\address[author1]{College of Mathematics and Informatics, South China Agricultural University, Guangzhou, 510640, China. Email:20212115003@stu.scau.edu.cn}
\address[author2]{ Corresponding author. College of Mathematics and Informatics, South China Agricultural University, Guangzhou, 510640, China. Email: hcl2016@scau.edu.cn} 
\address[author3]{Department of Mathematics, Faculty of Science and Technology, University of Macau, Macau, P. R. China. Email:kikou@umac.mo} 

\begin{abstract}
We examine how the square-integrable function subspaces are transformed using the holomorphic Fourier transform. On account of this, the extended Paley-Wiener theorem over the Hardy-Sobolev spaces is produced. The theorem also asserts that the reproducing kernel of the Hardy-Sobolev spaces can be found. We discuss the relationship between the disc and the upper half-plane.

\end{abstract}

\footnotetext{2020 Mathematics Subjuection Classification. Primary 46C07; Secondary 30A99, 47G10}



\end{frontmatter}


\section{Introduction}
\label{}
Function space, as a branch of functional analysis, has  been widely developed. Classical Hardy space has formed a rich theoretical system \cite{1,6}. It consists of functions in the Lebesgue space $L^p$ ($p > 1$) that are smooth inside the domain and have finite integral growth on the boundary.

\subsection{Related works}
S.L. Sobolev developed a novel kind of function space during his investigation of the elastic wave problem in the 1930s. This space later came to be known as Sobolev space, which is a Banach space maked up by weakly differentiable functions.

Due to the smoothness of Sobolev space functions, they are crucial to the study of mathematical analysis. In harmonic analysis and partial differential equations (PDE for short) theory, when $0<p\le1$, Hardy space is the best alternative to Lebesgue space $L^p$. When $p>1$, the Sobolev space requires that the weak derivative of the function belongs to $L^p$\cite{15}. 

As a result, it makes sense to research Hardy-Sobolev spaces, which demand that a function's derivative also belongs to Hardy spaces. As a significant analytic function space, Hardy-Sobolev space contains a number of well-known classical function spaces, including Dirichlet space, Hardy space, Bergmann space, and others. Since it has a more sophisticated spatial structure, it will have a more deep outcome than generic classical space, which seems sense to research the relative characteristics of this space.

Hardy-Sobolev spaces also combine the smoothness attributes of Sobolev space with the integral development characteristics of Hardy space. 
The functions have smoothness inside the domain and finite integral growth on the boundary, which makes the Hardy-Sobolev space particularly helpful for expressing issues with boundary influences, including boundary value problems. 

In a word, the Hardy-Sobolev spaces have wide applications in areas such as PDE, boundary value problems, harmonic analysis, and more. Exactly, it provides a suitable framework for handling functions subject to boundary conditions, and offers  tools for  problem-solving in these areas\cite{12,13,14}.

\subsection{Notations}
Let $\mathbb{R}^+$ be the interval (0,$\infty$) and $\mathbb{C}^+$,$\mathbb{C}_+$ respectively denote the upper half-plane and the right-hand half-plane of complex plane, i.e., $\mathbb{C}^+=\{z\in\mathbb{C} : \Im z>0\}$, $\mathbb{C}_+=\{z\in\mathbb{C} : \Re z>0\}$. We let $L^{p}(\mathbb{R}^+)$ be the usual Lebesgue space on (0,$\infty$) corresponding to its norm
$$\Vert f \Vert_p=\left(\int_{0}^{\infty}\vert f(t) \vert^{p}dt\right)^\frac{1}{p}<\infty,\qquad f\in L^{p}(\mathbb{R}^+).$$
In particular, $L^{2}(\mathbb{R}^+)$ is a Hilbert space with its inner product given by
$$\langle f,g \rangle_2=\int_{0}^{\infty}f(t) \overline{g(t)}dt,\qquad f,g\in L^{2}(\mathbb{R}^+). $$

We will usually write $\langle \cdot,\cdot \rangle$ to denote the inner product of Hilbert space. For $1\le p<\infty$ the space $\mathcal{T}^{(n)}_p$ consists of all the measurable functions on $\mathbb{R}^+$ which belongs to the closure of $C^{\infty}_{c}(\mathbb{R}^+)$ corresponding to the norm
$$\Vert f \Vert_{\mathcal{T}^{(n)}_p}=\left(\int_{0}^{\infty}\vert f^{(n)}(t)t^n \vert^p dt\right)^\frac{1}{p}<\infty,\qquad f\in C^{\infty}_{c}(\mathbb{R}^+).$$
In particular, the space $\mathcal{T}^{(n)}_1$, a subalgebra of $L^{1}(\mathbb{R}^+)$ , is a convolution Banach algebra with the norm $\Vert \cdot \Vert_{\mathcal{T}^{(n)}_1}$. It was first introduced to approach Cauchy problems in \cite{7} and it has been studied in \cite{8,9} recently. 

When it turns to consider $p=2$, it was proved that the range $\mathcal{L}(\mathcal{T}^{(n)}_2)$, where $\mathcal{L}$ denotes to the Laplace transform given by
$$\mathcal{L}(f)(z)=\int_{0}^{\infty}f(t)e^{-zt}dt,\qquad f\in L^{2}(\mathbb{R}^+), z\in \mathbb{C}_+ ,$$
is an isometric isomorphism onto the Hardy-Sobolev space $H^{(n)}_2 (\mathbb{C}_+)$ of all the holomorphic functions $F$ in $H_2 (\mathbb{C}_+)$ satisfying condition $z^k F^{(k)}\in H_2 (\mathbb{C}_+)$ for every $k=0,1,\dots,n$ \cite{4,5}. Here the notation $H_2 (\mathbb{C}_+)$ denotes the classical Hardy space of all analytic functions $F$ over the right-hand half-plane $\mathbb{C}_+$ such that
$$\Vert F \Vert_{H_2 (\mathbb{C}_+)}=\sup_{x>0} \left(\frac{1}{2\pi}\int_{-\infty}^{+\infty}\vert F(x+iy) \vert^2 dy\right)^\frac{1}{2}<\infty.$$

Naturally, it inspires us to consider the Holomorphic Fourier transform $\mathcal{F}$ acting on the space $\mathcal{T}^{(n)}_2$, where $\mathcal{F}$ is given by
$$\mathcal{F}(f)(z)=\int_{0}^{\infty}f(t)e^{itz}dt,\qquad f\in L^{2}(\mathbb{R}^+),z\in \mathbb{C}^+.$$
Note that the Holomorphic Fourier transform $\mathcal{F}:L^{2}(\mathbb{R}^+) \to H_2 (\mathbb{C}^+)$ is also an isometric isomorphism. We will show that the range $\mathcal{F}(\mathcal{T}^{(n)}_2)$ is characterized as the Hilbert space $H^{(n)}_2 (\mathbb{C}^+)$, where the space $H^{(n)}_2 (\mathbb{C}^+)$ and $H_2 (\mathbb{C}^+)$ is similar to the definition above that the domain becomes to the upper half-plane $\mathbb{C}^+$.

Point evalutation functionals on $H_2 (\mathbb{C}^+)$ are bounded so that it is a Reproducing Kernel Hilbert Space(RKHS). In fact, its reproducing kernel $K(z,w)=K_{w}(z)$ is given by 
$$K_{w}(z)=\frac{i}{z-\overline{w}},\qquad z,w\in \mathbb{C}^+.$$
The space $H^{(n)}_2 (\mathbb{C}^+)$ is also a RKHS and one of the main purposes in this paper is to determine the reproducing kernel of $H^{(n)}_2 (\mathbb{C}^+)$. In the rest of paper, we will sometimes write $H^{(n)}_2$ and $H_2$ for short to represent the Hardy-Sobolev spaces and Hardy space over the upper half-plane. 

Here are the key notations and acronyms used in this paper.
\begin{table}[htbp]
	\caption{Notations}
	\begin{center}
		\begin{tabular}{cc}
			\toprule
			Notation  & Description \\
			\midrule
			$\mathbb{R^+}$ & Interval (0,$\infty$) \\
			$\mathbb{C}^+$ & Upper half-plane of complex plane \\
			$\mathbb{D}$ & unit disc of complex plane \\
			$\Im$ & Imaginary part \\
			$H_2$ & Hardy space \\
			$\mathcal{T}_2^{(n)}$ & Lebesgue-Sobolev space \\
			$H^{(n)}_2$ & Hardy-Soboelv space \\
			$\Vert \cdot \Vert$ & Norm \\
			$\langle \cdot,\cdot \rangle$ & Inner product \\
			$\mathcal{F}$ & Holomorphic Fourier transform \\
			RKHS & Reproducing Kernel Hilbert Space \\
			\bottomrule
		\end{tabular}
	\end{center}
\end{table}
\subsection{Paper contributions}
In this research, we focus on the properties of Hardy-Sobolev spaces defined on the upper half-plane. Royo established the extended form of the Paley-Wiener theorem for Hardy-Sobolev spaces on the right-hand half-plane by applying the standard Paley-Wiener theorem \cite{5}. Matache enhanced Royo's proof by introducing the equivalent inner product and constructing the associated auxiliary function, and he used the extended Paley-Wiener theorem to determine the reproducing kernel of Hardy-Sobolev spaces \cite{4}. Based on the work of these two researchers, this paper extends the famous Paley-Wiener theorem to Hardy-Sobolev spaces on the upper half-plane and computes the reproducing kernel of this space.
Our key findings are as follows.
\begin{itemize}
	\item[A.] Let $\mathcal{F}$ be the Holomorphic Fourier transform $\mathcal{F}:L^{2}(\mathbb{R}^+) \to H_2 (\mathbb{C}^+)$. Then
	$$\mathcal{F}(\mathcal{T}^{(n)}_2)=H^{(n)}_2.$$
	That is, Holomorphic Fourier transform $\mathcal{F}$ is an isometric isomorphism from $\mathcal{T}_2^{(n)}$ onto $H_2^{(n)}$.
	\item[B.] Let $n$ be a positive integer. Then the function $K_n$ is the reproducing kernel of $H^{(n)}_2$,which is given by 
	$$K_n(z,w)=\frac{1}{((n-1)!)^2}\int_{0}^{1}\int_{0}^{1}(1-y)^{n-1}(1-x)^{n-1}\frac{i}{yz-x\overline{w}}dxdy,\qquad z,w\in\mathbb{C^+}.$$
	\item[C.] If $F\in H^{(n)}_2$, then $z^{k-1}F^{(k-1)},k=1,2,\dots,n$ can be extended continuously at all nonzero points on the real axis, including the point at infinity. 
\end{itemize}
The contributions of this work are summarized below.
\begin{itemize}
	\item[1.] We extend the Paley-Wiener theorem and derive a version of the theorem regarding the Holomorphic Fourier transform about Hardy-Sobolev spaces. Furthermore, we derive the reproducing kernel with respect to the integral form of Hardy-Sobolev spaces over the upper half-plane, which extends Hardy-Sobolev space theory.
	
	\item[2.]We establish the link between Hardy-Sobolev spaces in the upper half-plane and classical Hardy spaces on the disc, which provides additional tools and approaches for dealing with Hardy-Sobolev spaces.
	
\end{itemize}


\subsection{Paper outlines}
The paper is organized as follows. In Section 2, we define  spaces $L^2(t^n)$ and $\mathcal{T}^{(n)}_2$ with their corresponding norm, which are introduced in \cite{5}. They are connected by the integration operator $W^{-n}$. Then we give some propositions about these spaces. In section 3, we introduce the Holomorphic Fourier transform and some formulas
related to the transform and $n$-times derivation, which will be used to prove the extended Paley-Wiener theorem later. Then we give the definition of the space $H^{(n)}_2$ and finish the proof. Moreover, we compute the reproducing kernel of $H^{(n)}_2$ by applying the theorem. In addition, according to the Cayley transform, we give the connection of spaces  $H^{(n)}_2$ with the usual Hardy space $H_2(\mathbb{D})$ on the disc $\mathbb{D}$.

\section{Preliminary}
We use $W^{-1}$ to denote the integration operator on the space $C^{\infty}_{c}(\mathbb{R}^+)$ which is given by
$$W^{-1}\varphi(t)=\int_{t}^{\infty}\varphi(s)ds,\qquad \varphi\in C^{\infty}_{c}(\mathbb{R}^+),t\ge0.$$
For any function $\varphi\in C^{\infty}_{c}(\mathbb{R}^+)$, it is easy to show that $W^{-1}\varphi\in C^{\infty}_{c}(\mathbb{R}^+)$ and so that we can define the integration operator $W^{-n}=W^{-1}(W^{-(n-1)})$ for every positive integer $n$. By applying Fubini’s theorem we get
$$W^{-n}\varphi(t)=\frac{1}{(n-1)!}\int_{t}^{\infty}(s-t)^{n-1}\varphi(s)ds,\qquad \varphi\in C^{\infty}_{c}(\mathbb{R}^+),t\ge0.$$
In fact, the index $n$ can be extended to any positive real number $\alpha$, for more details about $W^{-\alpha}$, see \cite{5,16}.

Next we consider the operator $W^{-n}$ act on the Hilbert space $L^2(t^n)$.
\begin{defn}(\cite{5})
	For $n\in\mathbb{N}$, let $L^2(t^n)$ be the space of all complex measurable function $\varphi$ on $\mathbb{R}^+$ whose norm satisfies
	$$\Vert \varphi \Vert_{L^2 (t^n)}=\left(\int_{0}^{\infty}\vert \varphi(t)t^n \vert^2dt\right)^\frac{1}{2}<\infty.$$
\end{defn}
\begin{prop}(\cite{4})
	For $n\in\mathbb{N}$ and $\varphi\in L^2 (t^n)$. Then for $1\le k\le n$, we have 
	$$W^{-k}\varphi(t)=\frac{1}{(k-1)!}\int_{t}^{\infty}(s-t)^{n-1}\varphi(s)ds,\qquad t>0.$$
	Moreover, $W^{-k}\varphi$ is $(k-1)$-times differentiable with $(W^{-k}\varphi)^{(l)}=(-1)^l W^{-(k-l)}\varphi$ for every $1\le l\le k-1$, and $(-1)^k(W^{-k}\varphi)^{(k)}=\varphi$. 
\end{prop}

Take $m\in\mathbb{N}$ and Borel function $\varphi$, we get $W^{-n}\varphi\in L^{2}(\mathbb{R}^+)$ for all $\varphi\in L^2(t^n)$ by applying Hardy’s inequality \cite{2} 
$$\int_{0}^{\infty}\left(W^{-m}\varphi(t)\right)^2dt\le \left(\frac{\Gamma(\frac{1}{2})}{\Gamma(m+\frac{1}{2})}\right)^2\int_{0}^{\infty}(t^m\varphi(t))^2dt.$$
As a matter of fact, the integration operator $W^{-n}:L^2(t^n)\to L^{2}(\mathbb{R}^+)$ is injective, which enables us to define Hardy-Lebesgue space $\mathcal{T}^{(n)}_2$.
\begin{defn}(\cite{5})
	For $n\in\mathbb{N}$, let $\mathcal{T}^{(n)}_2$ be the range of operator $W^{-n}$ on $L^2 (t^n)$, i.e., 
	$$\mathcal{T}^{(n)}_2=W^{-n}(L^2 (t^n)),$$
	which is a subspace of $L^{2}(\mathbb{R}^+)$. It is clear that $\mathcal{T}^{(n)}_2$ is a Hilbert space with the inner product
	$$\langle f,g\rangle_{\mathcal{T}^{(n)}_2}=\int_{0}^{\infty}f^{(n)}(t)\overline{g^{(n)}(t)}t^{2n}dt,\qquad f,g\in\mathcal{T}^{(n)}_2.$$
\end{defn}
For $n\ge k\ge 0$, let $\varphi(t)=\vert f^{(n)}(t)\vert t^k$ and $m=n-k$, by applying Hardy’s inequality  again we get that the size of the space $\mathcal{T}^{(n)}_2$ decreases as $n$ increases, i.e.,
\begin{align}
	\mathcal{T}^{(n)}_2\subseteq \mathcal{T}^{(k)}_2\subseteq L^{2}(\mathbb{R}^+),\qquad n\ge k\ge 0.
\end{align}
Moreover, by using Holder’s inequality and taking $\varphi=(-1)^nf^{(n)}$ in Proposition 2.2 we obtain
\begin{align}
	\vert f^{(k)}(t)\vert\le C_{n,k}t^{-k-\frac{1}{2}}\Vert f\Vert_{\mathcal{T}^{(n)}_2},\qquad f\in\mathcal{T}^{(n)}_2,0\le k\le n-1,t>0,
\end{align}
where $C_{n,k}$ is a constant related to $n$ and $k$.
\begin{lem}(\cite{4}) 
	Take $n\ge 0$, the $(n+1)$-square matrix $C_n=(c_{i,j})_{0\le i,j\le n}$ defined by
	\begin{align*}
		c_{i,j}=\begin{cases}
			0,&i<j;\\
			\binom{i}{j}\frac{i!}{j!},&i\ge j.
		\end{cases}
	\end{align*}
	satisfies:
	
	(i) the matrix $C_n$ is invertible and $C^{-1}_n=((-1)^{i+j}c_{i,j})_{0\le i,j\le n}$.
	
	(ii) $\forall f\in\mathcal{T}^{(n)}_2,t>0,$ we have 
	\begin{flalign*}
		&(t^nf)^{(n)}(t)=\sum_{k=0}^{n}c_{n,k}t^kf^{(k)}(t),\\
		&t^nf^{(n)}(t)=\sum_{k=0}^{n}(-1)^{k+n}c_{n,k}(t^kf)^{(k)}(t).
	\end{flalign*}
\end{lem}

\begin{prop}(\cite{4})
	Take $f,g\in\mathcal{T}^{(n)}_2$. Then 
	$$\langle f,g\rangle_{\mathcal{T}^{(n)}_2}=\langle (t^nf)^{(n)},(t^ng)^{(n)}\rangle_2,$$
	where $\langle f,g \rangle_2=\int_{0}^{\infty}f(t) \overline{g(t)}dt$.
\end{prop}

\section{Main results}
Recall that $H_2(\mathbb{C^+})$ is the space of holomorphic function on the upper half-plane $\mathbb{C^+}$ such that (\cite{1,6})
$$\Vert f\Vert_{H_2}=\sup_{y>0}\frac{1}{2\pi}\int_{-\infty}^{+\infty}\vert f(x+iy)\vert^2dx<\infty.$$
In fact, $H_2(\mathbb{C^+})$ is a Hilbert space with the inner product
$$\langle f,g\rangle_{H_2}=\frac{1}{2\pi}\int_{-\infty}^{+\infty}f^{*}(x)\overline{g^{*}(x)}dx,\qquad f,g\in H_2(\mathbb{C^+}),$$
where $f^{*}(x)=\lim\limits_{y\to 0}f(x+iy)$ a.e. $x\in\mathbb{R}$.

Let $f$ be any function in $L^2(\mathbb{R^+})$ and define the Holomorphic Fourier transform $\mathcal{F}$ given by
$$\mathcal{F}(f)(z)=F(z)=\int_{0}^{\infty}f(t)e^{itz}dt,\qquad z\in\mathbb{C^+},$$
If $z=x+iy\in\mathbb{C^+}$, then $\vert e^{itz}\vert=e^{-ty}$, which shows that the integral above exists as a Lebesgue integral. By applying a series of theorems about real and complex analysis ,we can prove that $F$ is a holomorphic function on $\mathbb{C^+}$ \cite{3}. The classical Paley-Wiener theorem (\cite{3}) shows that the Holomorphic Fourier transform
$\mathcal{F}:L^{2}(\mathbb{R}^+) \to H_2 (\mathbb{C}^+)$ is an isometric isomorphism, i.e., $F\in H_2(\mathbb{C^+})$ if and only if there exists an $f\in L^2(\mathbb{R^+})$ such that $F=\mathcal{F}(f)$ and 
$$\Vert F\Vert_{H_2}=\Vert f \Vert_2.$$

Next we will consider the Paley-Wiener type theorem on the Hardy-Sobolev spaces over the upper half-plane.
\begin{lem}
	$\forall n\in\mathbb{N},z\in\overline{C}^+\backslash\{0\}$ and $f\in\mathcal{T}^{(n)}_2,$
	\begin{flalign*}
		&(-1)^kz^k[\mathcal{F}(f)]^{(k)}(z)=\sum_{j=0}^{k}c_{k,j}\mathcal{F}(t^jf^{(j)})(z),\qquad k=0,1,\dots,n;\\
		&(-1)^k\mathcal{F}(t^kf^{(k)})(z)=\sum_{j=0}^{k}c_{k,j}z^j[\mathcal{F}(f)]^{(j)}(z),\qquad k=0,1,\dots,n,
	\end{flalign*}
	where $c_{k,j}$ is a constant defined by Lemma 2.4.
\end{lem}
\begin{proof}
	Let $j,k$ be nonnegative integers such that $0\le j\le k\le n$. Taking $h=t^kf$ and integrating by parts $k$ times, we get
	$$\mathcal{F}(h^{(k)})(z)=(-iz)^{k}\mathcal{F}(h)(z)-\sum_{j=0}^{k-1}(-iz)^{k-1-j}h^{(j)}(0).$$
	By applying Leibniz’s rule of derivation in $(t^kf)^{(j)}$ and the formula given by (1), we get
	$$(t^kf)^{(j)}(0)=\lim\limits_{t\to 0^+}(t^kf)^{(j)}(t)=0,\qquad j=0,1,\dots,k-1.$$
	Thus we have 
	$$\mathcal{F}((t^kf)^{(k)})(z)=(-iz)^k\mathcal{F}(t^kf)(z)=(-1)^kz^k[\mathcal{F}(f)]^{(k)}(z).$$
	Then
	$$(-1)^kz^k[\mathcal{F}(f)]^{(k)}(z)=\mathcal{F}((t^kf)^{(k)})(z)=\sum_{j=0}^{k}c_{k,j}\mathcal{F}(t^jf^{(j)})(z)$$
	where the second equality is based on Lemma 2.4.
	
	By applying the invertibility of the matrix defined by Lemma 2.4, we have
	$$\mathcal{F}(t^jf^{(j)})(z)=\sum_{k=0}^{j}(-1)^jc_{j,k}z^k[\mathcal{F}(f)]^{(k)}(z).$$
	Interchange the position of $j$ and $k$, we derive
	$$(-1)^k\mathcal{F}(t^kf^{(k)})(z)=\sum_{j=0}^{k}c_{k,j}z^j[\mathcal{F}(f)]^{(j)}(z).$$
\end{proof}
Now we define Hardy-Sobolev space over the upper half-plane.
\begin{defn}
	For $n\in\mathbb{N}$, let $H^{(n)}_2$ be the linear space of all holomorphic functions $F$ on $\mathbb{C^+}$ such that
	$$z^kF^{(k)}\in H_2,\qquad k=0,1,\dots,n.$$
	It is clear that $H^{(n)}_2\subseteq H^{(0)}_2=H_2$.
\end{defn}
Here is the Paley-Wiener theorem on Hardy-Sobolev spaces.
\begin{thm}
	Let $\mathcal{F}$ be the Holomorphic Fourier transform $\mathcal{F}:L^{2}(\mathbb{R}^+) \to H_2 (\mathbb{C}^+)$. Then
	$$\mathcal{F}(\mathcal{T}^{(n)}_2)=H^{(n)}_2.$$
	Thus, $H^{(n)}_2$ is a Hilbert space endowed with the inner product
	$$\langle F,G \rangle_{H^{(n)}_2}=\int_{\mathbb{-\infty}}^{+\infty}x^{2n}(F^{*})^{(n)}(x)\overline{(G^{*})^{(n)}(x)}dx,\qquad F,G\in H^{(n)}_2,$$
	and corresponding to the norm
	$$\Vert F\Vert_{H^{(n)}_2}=\Vert z^nF^{(n)}\Vert_{H_2},\qquad F\in H^{(n)}_2.$$
	Moreover, $\mathcal{F}$ is an isometric isomorphism from $\mathcal{T}^{(n)}_2$ onto $H^{(n)}_2$, i.e., $F,G\in H^{(n)}_2$ if and only if there exists $f,g\in\mathcal{T}^{(n)}_2$ and have the property
	$$\langle f,g \rangle_{\mathcal{T}^{(n)}_2}=\langle \mathcal{F}(f),\mathcal{F}(g)\rangle_{H^{(n)}_2}.$$
\end{thm}
\begin{proof}
	Suppose $f\in\mathcal{T}^{(n)}_2$, then for $0\le k\le n$, it is easy to see that $t^kf^{(k)}\in L^2(\mathbb{R^+})$. 
	
	By applying Lemma 3.1 and the classical Paley-Wiener Theorem, we derive that $z^k[\mathcal{F}(f)]^{(k)}\in H_2$. Thus $\mathcal{F}(f)\in H^{(n)}_2$.
	
	Conversely, provided that $F\in H^{(n)}_2\subseteq H_2$, then according to the classical Paley-Wiener theorem, there exists a function $f\in L^2(\mathbb{R^+})$ such that $F=\mathcal{F}(f)$. 
	
	All we need is to prove that $f\in\mathcal{T}^{(n)}_2$. To arrive this aim, considering the function
	$$G(z)=(-1)^n\sum_{j=0}^{n}c_{n,j}z^jF^{(j)}(z),\qquad z\in\mathbb{C^+},$$
	we get $G\in H_2$ since $F\in H^{(n)}_2$.
	
	Therefore, there exists a function $g\in L^2(\mathbb{R^+})$ such that $G=\mathcal{F}(g)$. What's more, we define another function
	$$h(s)=-\frac{1}{(n-1)!}\int_{s}^{\infty}(s-t)^{n-1}\frac{g(t)}{t^n}dt.$$
	
	Note that $\frac{g(t)}{t^n}\in L^2(t^n)$, it is clear that it belongs to $\mathcal{T}^{(n)}_2$. By applying proposition 2.2, we derive $g=t^nh^{(n)}$.
	
	Now according to Lemma 3.1, we have
	$$G(z)=\mathcal{F}(g)(z)=\mathcal{F}(t^nh^{(n)})(z)=(-1)^n\sum_{j=0}^{n}c_{n,j}z^j[\mathcal{F}(h)]^{(j)}(z),\qquad z\in\mathbb{C^+}.$$
	Therefore, since $F=\mathcal{F}(f)$, we get
	$$(-1)^n[z^n\mathcal{F}(h-f)]^{(n)}(z)=(-1)^n\sum_{j=0}^{n}c_{n,j}z^j[\mathcal{F}(h-f)]^{(j)}(z)=0.$$
	It shows that $z^n\mathcal{F}(h-f)=P_n(z)$ where $P_n(z)$ denotes the polynomial with its degree less than or equal to $n-1$.
	
	Note that $P_n(z)z^{-n}=\mathcal{F}(h-f)\in H_2$, we get $\mathcal{F}(h-f)=0$. Then by the injectivity of $\mathcal{F}$, we derive $f=h\in\mathcal{T}^{(n)}_2$. Therefore, we have shown that $\mathcal{F}(\mathcal{T}^{(n)}_2)=H^{(n)}_2$.
	
	Take $f,g\in\mathcal{T}^{(n)}_2$, then $F=\mathcal{F}(f),G=\mathcal{F}(g)$ lies in $H^{(n)}_2$. By proposition 2.5, we have
	\begin{align*}
		\langle f,g \rangle_{\mathcal{T}^{(n)}_2}&=\langle(t^nf)^{(n)},(t^ng)^{(n)}\rangle_2\\
		&=\langle \mathcal{F}[(t^nf)^{(n)}],\mathcal{F}[(t^ng)^{(n)}]\rangle_{H_2}\\
		&=\langle (-1)^nz^n[\mathcal{F}(f)]^{(n)},(-1)^nz^n[\mathcal{F}(g)]^{(n)}\rangle_{H_2}\\
		&=\langle F,G\rangle_{H^{(n)}_2}.
	\end{align*}
\end{proof}
Let $X$ be a set, recall that $\mathcal{H}$ is a Reproducing Kernel Hilbert Space or, more briefly, an RKHS on $X$, if it is a Hilbert space of all complex functions on $X$ with its inner product $\langle \cdot,\cdot\rangle$ satisfying the condition that the evaluation functional, $E_x:\mathcal{H}\to\mathbb{C}$, defined by $E_x(f)=f(x)$, is bounded for every $x\in X$.

If $\mathcal{H}$ is an RKHS on $X$, then an application of the Riesz representation theorem shows that the linear evaluation functional is given by the inner product with a unique vector in $\mathcal{H}$. 

Thus, for each $x\in X$, there exists a unique vector $k_x\in\mathcal{H}$ such that for every $f\in\mathcal{H}, f(x)=E_x(f)=\langle f,k_x\rangle$. The function $k_x$ is called kernel-functions and the span generated by all of kernel-functions is dense in $\mathcal{H}$.

The function $K:X\times X\to\mathbb{C}$ defined by $K(x,y)=k_y(x)=\langle k_y,k_x\rangle,x,y\in X$ is said to be the reproducing kernel of $\mathcal{H}$\cite{10}. Point evaluation functions on $H^{(n)}_2$ are continuous so that $H^{(n)}_2$ is a RKHS.

Next we determine the reproducing kernel of $H^{(n)}_2$.
\begin{thm}
	Let $n$ be a positive integer. Then the function $K_n$ defined on $\mathbb{C^+}\times\mathbb{C^+}$ by 
	$$K_n(z,w)=\frac{1}{((n-1)!)^2}\int_{0}^{1}\int_{0}^{1}(1-y)^{n-1}(1-x)^{n-1}\frac{i}{yz-x\overline{w}}dxdy,\qquad z,w\in\mathbb{C^+},$$
	is the reproducing kernel of $H^{(n)}_2$, i.e., if $K_{n,w}(z)=K_n(z,w)$ then
	$$K_{n,w}\in H^{(n)}_2\text{ and }F(w)=\langle F,K_{n,w}\rangle_{H^{(n)}_2},\forall F\in H^{(n)}_2.$$
\end{thm}
\begin{proof}
	Take $w\in\mathbb{C^+}$ and $n\in \mathbb{N^*}$. Let
	$$\varphi_{w,n}(t)=\frac{1}{t^n}\int_{0}^{1}\frac{(1-x)^{n-1}}{(n-1)!}e^{-i\overline{w}tx}dx,\qquad t>0.$$
	We will show that $\varphi_{w,n}\in L^2(t^n)$. In fact,
	\begin{align*}
		\left(\int_{0}^{\infty}\vert\varphi_{w,n}(t)\vert^2t^{2n}dt\right)^\frac{1}{2}&\le \left(\int_{0}^{\infty}\left(\int_{0}^{1}\frac{(1-x)^{n-1}}{(n-1)!}e^{-(\Im w)tx}dx\right)^2dt\right)^\frac{1}{2}\\
		&\le\int_{0}^{1}\frac{(1-x)^{n-1}}{(n-1)!}\left(\int_{0}^{\infty}e^{-2(\Im w)tx}\right)^\frac{1}{2}dx\\
		&=\frac{\sqrt{\frac{\pi}{2}}}{\Gamma(n+\frac{1}{2})\sqrt{\Im w}}.
	\end{align*}
	Hence, $g_{w,n}(t)=W^{-n}\varphi_{w,n}(t)\in\mathcal{T}^{(n)}_2$ and
	$$g_{w,n}(t)=\frac{1}{((n-1)!)^2}\int_{t}^{\infty}\frac{(s-t)^{n-1}}{s^n}\int_{0}^{1}(1-x)^{n-1}e^{-i\overline{w}sx}dxds.$$
	By proposition 2.1, we have
	$$g_{w,n}^{(n)}(t)=(-1)^n\varphi_{w,n}(t)=\frac{(-1)^n}{t^n}\int_{0}^{1}\frac{(1-x)^{n-1}}{(n-1)!}e^{-i\overline{w}tx}dx=\frac{(-1)^n}{t^{2n}}\int_{0}^{t}\frac{(t-s)^{n-1}}{(n-1)!}e^{-i\overline{w}s}ds.$$
	
	Take $F\in H^{n}_2$. By Theorem 3.3, there exists a function $f\in\mathcal{T}^{(n)}_2$ such that $F=\mathcal{F}(f)$. 
	
	Then by applying Fubini's theorem, we have $\forall f\in\mathcal{T}^{(n)}_2$,
	\begin{align*}
		F(w)=\mathcal{F}(f)(w)=\int_{0}^{\infty}f(t)e^{iwt}dt&=(-1)^n\int_{0}^{\infty}\int_{t}^{\infty}\frac{(s-t)^{n-1}}{(n-1)!}f^{(n)}(s)ds\ e^{iwt}dt\\
		&=\int_{0}^{\infty}f^{(n)}(t)\int_{0}^{t}(-1)^n\frac{(t-s)^{n-1}}{(n-1)!}e^{iws}ds\\
		&=\int_{0}^{\infty}f^{(n)}(t)t^{2n}\overline{g^{(n)}_{w,n}(t)}dt\\
		&=\langle f,g_{w,n}\rangle_{\mathcal{T}^{(n)}_2}.
	\end{align*}
	On the other hand, use the change of variables $s=\frac{t}{y}$, we get
	$$g_{w,n}(t)=\int_{0}^{1}\int_{0}^{1}\frac{(1-y)^{n-1}(1-x)^{n-1}}{y[(n-1)!]^2}e^{\frac{-i\overline{w}tx}{y}}dxdy.$$
	Then for $z\in\mathbb{C^+}$,
	\begin{align*}
		\mathcal{F}(g_{w,n})(z)&=\int_{0}^{1}\int_{0}^{1}\frac{(1-y)^{n-1}(1-x)^{n-1}}{y[(n-1)!]^2}\mathcal{F}(e^{\frac{-i\overline{w}tx}{y}})(z)dxdy\\
		&=\int_{0}^{1}\int_{0}^{1}\frac{(1-y)^{n-1}}{(n-1)!}\frac{(1-x)^{n-1}}{(n-1)!}\frac{i}{yz-x\overline{w}}dxdy\\
		&=K_{n,w}(z).
	\end{align*}
	Therefore, by theorem 3.3, we have 
	$$F(w)=\langle f,g_{w,n}\rangle_{\mathcal{T}^{(n)}_2}=\langle F,K_{n,w}\rangle_{H^{(n)}_2},\qquad w\in\mathbb{C^+}.$$
\end{proof}

Recall that the usual Hardy space $H_2(\mathbb{D})$ on the unit disc $\mathbb{D}=\{z\in\mathbb{C}:\vert z\vert=1\}$ is endowed with the norm
$$\Vert F\Vert_{H_2(\mathbb{D})}=\sup_{0<r<1}\left(\frac{1}{2\pi}\int_{0}^{2\pi}\vert F(re^{i\theta})\vert^2d\theta\right)^\frac{1}{2},\qquad F\in H_2(\mathbb{D}).$$
The Cayley transform $\Phi(\lambda)=i\frac{1-\lambda}{1+\lambda},\lambda\in\mathbb{D}$ maps the unit disc $\mathbb{D}$ conformally onto the upper half-plane $\mathbb{C^+}$. Then $F\in H_2(\mathbb{C^+})$ if and only if $F_{\mathbb{D}}\in (1+\lambda)H_2(\mathbb{D})$ (\cite{6}), where
\begin{align}
	F_{\mathbb{D}}(\lambda)=F\left(i\frac{1-\lambda}{1+\lambda}\right),\qquad \lambda\in\mathbb{D}.
\end{align}
\begin{prop}
	Supposed $F$ is an holomorphic function on $\mathbb{C^+}$ and $F_{\mathbb{D}}$ is the function on the unit disc given by (3). Then $F\in H^{(1)}_2$ if and only if
	$$F_{\mathbb{D}}\in(1+\lambda)H_2(\mathbb{D})\text{\qquad and\qquad  }F^{\prime}_{\mathbb{D}}\in(1-\lambda)^{-1}H_2(\mathbb{D})$$
\end{prop}
\begin{proof}
	The condition $F\in H^{(1)}_2$ is equivalent to $F\in H_2(\mathbb{C^+})$ and $zF^{\prime}\in H_2(\mathbb{C^+})$. Then we have
	\begin{align}
		F_{\mathbb{D}}\in(1+\lambda)H_2(\mathbb{D})\text{\qquad and\qquad  }i\frac{1-\lambda}{1+\lambda}F^{\prime}\left(i\frac{1-\lambda}{1+\lambda}\right)\in(1+\lambda)H_2(\mathbb{D}).
	\end{align}
	Note that 
	$$F^{\prime}_{\mathbb{D}}(\lambda)=\frac{-2i}{(1+\lambda)^2}F^{\prime}\left(i\frac{1-\lambda}{1+\lambda}\right).$$
	we get 
	$$F_{\mathbb{D}}^{\prime}(\lambda)\in(1-\lambda)^{-1}H_2(\mathbb{D}).$$
\end{proof}
Of course Proposition 4.1 can be extended for $n>1$ in the same way so that we could deal with Hardy-Sobolev spaces on $\mathbb{D}$ instead on the upper half-plane.
\begin{prop}
	If $F\in H^{(n)}_2$, then $z^{k-1}F^{(k-1)},k=1,2,\dots,n$ can be extended continuously at all nonzero point on the real axis and the point at infinity is included.
\end{prop}
\begin{proof}
	Take $F\in H^{(1)}_2$.  By proposition 4.1, we get
	$$F_{\mathbb{D}}\in(1+\lambda)H_2(\mathbb{D})\subseteq H_2(\mathbb{D})$$
	and 
	$$(1-\lambda)F^{\prime}_{\mathbb{D}}\in H_2(\mathbb{D}).$$
	
	Note that
	$$[(1-\lambda)F_{\mathbb{D}}(\lambda)]^{\prime}=-F_{\mathbb{D}}(\lambda)+(1-\lambda)F_{\mathbb{D}}^{\prime}(\lambda)\subseteq H_2(\mathbb{D})\subseteq H_1(\mathbb{D}).$$
	
	Then, the function $(1-\lambda)F_{\mathbb{D}}(\lambda)$ can be extended continuously on the unit circle, that is, $F_{\mathbb{D}}$ extends by continuity to the unit circle except possibly at 1.
	
	Since the Cayley transform maps 1 to 0 and -1 to $\infty$, the conclusion on the extensibility of $F$ by continuity follows.
	
	Now take $F\in H^{(n)}_2,n>1$.Note that
	$$z\frac{d}{dz}\left(z^{k-1}F^{(k-1)}(z)\right)=(k-1)z^{k-1}F^{(k-1)}(z)+z^kF^{k}(z),\qquad z\in\mathbb{C^+}.$$
	we get $z^{k-1}F^{(k-1)}(z)$ belongs to $H^{(1)}_2$ for $1\le k\le n$. 
	
	Therefore, the continutiy on the real axis of the functions $z^{k-1}F^{(k-1)},k=1,2,\dots,n$ follows by the previous discussion.
\end{proof}

\section{Conclusion}
In this study, we offer an extended Paley-Wiener theorem for Hardy-Sobolev spaces $H^{(n)}_2$ and use it to find the reproducing kernel $K_n$ of Hardy-Sobolev space on the upper half-plane. In addition, using the Cayley transform, we investigate the boundary behavior of functions in Hardy-Sobolev spaces.

\section*{Acknowledgements}
This work was supported by the NSF of China (Grant No. 12226318 and 12071155). This work was also partly supported by The Science and Technology Development Fund, Macau SAR (File no. FDCT/0036/2021/AGJ) and University of Macau (File no. MYRG2022-00108-FST).

\section*{Data availability}
Data sharing not applicable to this article as no datasets were generated or analysed during the current study.

\section*{Declarations}
The author declare no conflict of interest.





\bibliographystyle{unsrt}
\bibliography{reference}

\begin{thebibliography}{10}

\bibitem{1}
Peter~L Duren.
\newblock {\em Theory of H p Spaces}.
\newblock Academic press, 1970.

\bibitem{6}
Javad Mashreghi.
\newblock {\em Representation theorems in Hardy spaces}, volume~74.
\newblock Cambridge University Press, 2009.

\bibitem{15}
Robert~A Adams and John~JF Fournier.
\newblock {\em Sobolev spaces}.
\newblock Elsevier, 2003.

\bibitem{12}
Ha{\"\i}m Brezis and Petru Mironescu.
\newblock Gagliardo-nirenberg, composition and products in fractional sobolev spaces.
\newblock {\em Journal of Evolution Equations}, 1(4):387--404, 2001.

\bibitem{13}
Eleonora Di~Nezza, Giampiero Palatucci, and Enrico Valdinoci.
\newblock Hitchhiker's guide to the fractional sobolev spaces.
\newblock {\em Bulletin des sciences math{\'e}matiques}, 136(5):521--573, 2012.

\bibitem{14}
Akihiko Miyachi.
\newblock Hardy-sobolev spaces and maximal functions.
\newblock {\em Journal of the Mathematical Society of Japan}, 42(1):73--90, 1990.

\bibitem{7}
Wolfgang Arendt.
\newblock Vector-valued laplace transforms and cauchy problems.
\newblock {\em Israel Journal of Mathematics}, 59:327--352, 1987.

\bibitem{8}
Jose~E Gale and Luis Sanchez-Lajusticia.
\newblock A sobolev algebra of volterra type.
\newblock {\em Journal of the Australian Mathematical Society}, 92(3):313--334, 2012.

\bibitem{9}
Jos{\'e}~E Gal{\'e}, Pedro~J Miana, and Juan~J Royo.
\newblock Estimates of the laplace transform on convolution sobolev algebras.
\newblock {\em Journal of approximation theory}, 164(1):162--178, 2012.

\bibitem{4}
Jos{\'e}~E Gal{\'e}, Valentin Matache, Pedro~J Miana, and Luis S{\'a}nchez-Lajusticia.
\newblock Hilbertian hardy-sobolev spaces on a half-plane.
\newblock {\em Journal of Mathematical Analysis and Applications}, 489(1):124131, 2020.

\bibitem{5}
Juan Jos{\'e}~Royo Espallargas.
\newblock {\em {\'A}lgebras y m{\'o}dulos de convoluci{\'o}n sobre r+ definidos mediante derivaci{\'o}n fraccionaria}.
\newblock PhD thesis, Universidad de Zaragoza, 2008.

\bibitem{16}
Stefan~G Samko.
\newblock Fractional integrals and derivatives.
\newblock {\em Theory and applications}, 1993.

\bibitem{2}
Godfrey~Harold Hardy, John~Edensor Littlewood, George P{\'o}lya, Gy{\"o}rgy P{\'o}lya, et~al.
\newblock {\em Inequalities}.
\newblock Cambridge university press, 1952.

\bibitem{3}
Walter Rudin.
\newblock Real and complex analysis. 1987.
\newblock {\em Cited on}, 156:16, 1987.

\bibitem{10}
Vern~I Paulsen and Mrinal Raghupathi.
\newblock {\em An introduction to the theory of reproducing kernel Hilbert spaces}, volume 152.
\newblock Cambridge university press, 2016.

\end{thebibliography}

\end{document}